\documentclass[a4paper]{amsart}
\usepackage[T1]{fontenc}
\usepackage[utf8]{inputenc}
\usepackage{lmodern}

\usepackage{enumitem}
\setlist[enumerate,1]{label=\textup{(\arabic*)}}

\usepackage{mathtools, amssymb}

\usepackage{microtype}

\usepackage{hyperref}
\hypersetup{%
  bookmarksnumbered,
  bookmarksopen,
  breaklinks,
  pdfstartview=FitH,
  pdftitle={Cuspidal representations of reductive p-adic groups are relatively injective and projective},
  pdfauthor={Ralf Meyer},
  pdfsubject={Mathematics},
}
\usepackage[lite]{amsrefs}
\newcommand*{\MRref}[2]{ \href{http://www.ams.org/mathscinet-getitem?mr=#1}{MR \textbf{#1}}}

\theoremstyle{plain}
\newtheorem{theorem}{Theorem}[section]
\newtheorem{proposition}[theorem]{Proposition}
\newtheorem{lemma}[theorem]{Lemma}
\newtheorem{corollary}[theorem]{Corollary}

\theoremstyle{definition}
\newtheorem{definition}[theorem]{Definition}

\theoremstyle{remark}
\newtheorem{remark}[theorem]{Remark}

\DeclareMathOperator{\Hom}{Hom}
\DeclareMathOperator{\im}{im}

\newcommand{\twolinesubscript}[2]{\genfrac{}{}{0pt}{}{#1}{#2}}

\newcommand*{\nb}{\nobreakdash}

\newcommand*{\defeq}{\mathrel{\vcentcolon=}}
\newcommand*{\into}{\rightarrowtail}
\newcommand*{\prto}{\twoheadrightarrow}
\newcommand*{\blank}{\textup\textvisiblespace}

\newcommand*{\av}[1]{\langle#1\rangle}

\newcommand*{\N}{\mathbb N}
\newcommand*{\Q}{\mathbb Q}

\newcommand*{\Gl}{\textup{Gl}}
\newcommand*{\Id}{\textup{Id}}

\newcommand*{\Hecke}{\mathcal H}
\newcommand*{\Rep}{\mathfrak{Rep}}

\newcommand*{\Gz}[1][G]{{}^0 #1}
\newcommand*{\StabS}{G}

\newcommand*{\Buil}{\mathcal{BT}}

\newcommand*{\Pco}{F}

\begin{document}
\title[Cuspidal representations of reductive groups]{Cuspidal
  representations of reductive p-adic groups are relatively
  injective and projective}

\author{Ralf Meyer}
\email{rmeyer2@uni-goettingen.de}
\address{Mathematisches Institut\\
  Georg-August-Universit\"at G\"ottingen\\
  Bunsenstra\ss e 3--5\\
  37073 G\"ottingen\\
  Germany}

\begin{abstract}
  Cuspidal representations of a reductive \(p\)\nb-adic group~\(G\)
  over a field of characteristic different from~\(p\) are relatively
  injective and projective with respect to extensions that split by a
  \(U\)\nb-equivariant linear map for any subgroup~\(U\) that is
  compact modulo the centre.  The category of smooth representations
  over a field whose characteristic does not divide the pro-order
  of~\(G\) is the product of the subcategories of cuspidal
  representations and of subrepresentations of direct sums of
  parabolically induced representations.
\end{abstract}

\maketitle

\section{Introduction}
\label{sec:intro}

Let~\(G\) be a reductive linear algebraic group over a non-Archimedean
local field with residue field of characteristic~\(p\); we briefly
call~\(G\) a \emph{reductive \(p\)\nb-adic group}.  Let~\(R\) be a
commutative ring in which~\(p\) is invertible.  Let \(\Rep_R(G)\) be
the category of smooth representations of~\(G\) on \(R\)\nb-modules.

A smooth representation is \emph{cuspidal} if it is killed by the
parabolic restriction functors for all proper parabolic subgroups.  We
call an extension \(V'\into V\prto V''\) in \(\Rep_R(G)\)
\emph{cmc-split exact} if it splits in \(\Rep_R(K)\) for any
subgroup~\(K\) of~\(G\) that is compact modulo the centre~\(Z(G)\)
of~\(G\) (cmc).  A smooth representation is \emph{cmc-projective} or
\emph{cmc-injective}, respectively, if it is projective or injective
with respect to cmc-split extensions.

\begin{theorem}
  \label{the:rel_projective_injective}
  Cuspidal representations are cmc-projective and cmc-injective.
\end{theorem}

We call~\(R\) a \emph{good field} if it is a field whose
characteristic~\(\ell\) does not divide the pro-order of~\(G\); that
is, \(\ell\) does not divide \([U_1: U_2]\) for any compact open
subgroups \(U_2\subseteq U_1\subseteq G\).  Our theorem implies:

\begin{theorem}
  \label{the:enough_projective_injective}
  Any cuspidal representation over a good field is a quotient of a
  cuspidal projective representation and contained in a cuspidal
  injective representation.
\end{theorem}

This implies that the category of smooth representations of~\(G\) over
a good field is the product of the subcategory of cuspidal
representations and the subcategory of representations that are
contained in a sum of parabolically induced representations.  This
splitting is a crucial part of the Bernstein decomposition.

Our main theorem follows quickly from the theory of support
projections in~\cite{Meyer-Solleveld:Resolutions}.  We recall the
relevant notation and results in Section~\ref{sec:support}.
Section~\ref{sec:rel_projective} proves the assertions involving
relative projectivity and applies it to show that parabolically
induced representations have no cuspidal subquotients if~\(R\) is a
good field.  Section~\ref{sec:rel_injective} proves the assertions
about relative injectivity.  This implies that any smooth
representation over a good field is a direct sum of a cuspidal
representation and a subrepresentation of a parabolically induced
representation.

\section{Support projections}
\label{sec:support}

Let~\(\Buil\) be the affine Bruhat--Tits building of~\(G\).  We
treat~\(\Buil\) as a partially ordered set of polysimplices, with the
relation \(\sigma\prec\tau\) if~\(\sigma\) is a facet of~\(\tau\).
The group~\(G\) acts on~\(\Buil\).  We denote the stabiliser of
\(\sigma\in\Buil\) by
\[
\StabS_\sigma \defeq \{g\in G\mid g\sigma=\sigma\};
\]
its elements may permute the vertices of~\(\sigma\) non-trivially.

The group~\(\StabS_\sigma\) is open and compact modulo the
centre~\(Z(G)\) of~\(G\).  Any subgroup that is compact modulo the
centre is contained in~\(\StabS_\sigma\) for some \(\sigma\in\Buil\)
because it fixes some point in the geometric realisation of~\(\Buil\).
Hence an extension is cmc-split exact (see
Definition~\ref{def:split_extension}) if and only if it splits
\(\StabS_\sigma\)\nb-equivariantly for each \(\sigma\in\Buil\).

The normalised Haar measure on a compact, open, pro-\(p\) subgroup
\(U\subseteq G\) gives an idempotent element~\(\av{U}\) in the Hecke
algebra \(\Hecke = \Hecke(G,R)\) of~\(G\) with coefficients in~\(R\)
because \(p^{-1}\in R\).  Schneider and
Stuhler~\cite{Schneider-Stuhler:Rep_sheaves} and
Vign\'eras~\cite{Vigneras:Cohomology} use certain compact, open,
pro-\(p\) subgroups~\(U^n_\sigma\) for \(\sigma\in\Buil\),
\(n\in\N\) to construct resolutions for smooth representations
of~\(G\).  Let \(e_\sigma^n\defeq \av{U^n_\sigma}\).  For fixed
\(\sigma\in \Buil\), the groups \((U^n_\sigma)_{n\in\N}\) form a
neighbourhood basis of the unit in~\(G\).  Hence
\((e^n_\sigma)_{n\in\N}\) is an approximate unit in~\(\Hecke\).

\begin{theorem}
  \label{the:support_projections}
  Let \(\Sigma\subseteq\Buil\) be a finite, convex subcomplex.
  Define
  \[
  u^n_\Sigma
  \defeq \sum_{\sigma\in\Sigma} (-1)^{\dim(\sigma)} e^n_\sigma
  \in \Hecke.
  \]
  This element is idempotent and acts on any smooth representation
  such that
  \begin{equation}
    \label{eq:support_projection}
    \im(u^n_\Sigma) = \sum_{x\in\Sigma^0} \im(e^n_x),\qquad
    \ker(u^n_\Sigma) = \bigcap_{x\in\Sigma^0} \ker(e^n_x).
  \end{equation}
\end{theorem}

\begin{proof}
  The idempotents \((e_x^n)_{x\in\Buil^0}\) satisfy the conditions
  in \cite{Meyer-Solleveld:Resolutions}*{Definition 2.1}, and
  \(e^n_\sigma = \prod_{x\prec \sigma} e^n_x\) for
  \(\sigma\in\Buil\).  Hence everything follows from
  \cite{Meyer-Solleveld:Resolutions}*{Theorem 2.12}.
\end{proof}

The idempotent~\(u^n_\Sigma\) is called the \emph{support
  projection} of~\(\Sigma\) in~\cite{Meyer-Solleveld:Resolutions}.

\section{Relative projectivity of cuspidal representations}
\label{sec:rel_projective}

\begin{definition}
  A representation \(V\in\Rep_R(G)\) is \emph{cuspidal} if it is
  killed by the parabolic restriction functor for any proper parabolic
  subgroup.
\end{definition}

\begin{proposition}[\cite{Vigneras:l-modulaires}*{II.2.7}]
  \label{pro:cuspidal_criterion}
  A representation~\(V\) is cuspidal if and only if for each compact
  open pro-\(p\) subgroup \(U\subseteq G\) and each \(v\in V\), the
  set of \(g\in G\) with \(\av{U} g v\neq0\) is compact
  modulo~\(Z(G)\).
\end{proposition}

Since \(e^n_\sigma= U^n_\sigma\) and the
subgroups~\(U^n_\sigma\) form a neighbourhood basis in~\(G\), \(V\)
is cuspidal if and only if the set of \(g\in G\) with \(e^n_\sigma g
v\neq0\) is compact modulo the centre of~\(G\) for all \(n\in\N\),
\(\sigma\in\Buil\), \(v\in V\).  The set of these \(g\in G\) is closed
under multiplication by elements of the centre~\(Z(G)\) because
\(e^n_\sigma z g v = z e^n_\sigma g v = 0\) if and only if
\(e^n_\sigma g v=0\); and it is closed under left multiplication by
elements of~\(U^n_\sigma\).  Thus being cuspidal means that for each
\(n,\sigma,v\), the set of \(g\in G\) with \(e^n_\sigma g v\neq0\)
consists of only finitely many cosets~\(U^n_\sigma Z(G) g\).

\begin{definition}
  \label{def:split_extension}
  An \emph{extension} \(W'\into W\prto W''\) in \(\Rep_R(G)\) consists
  of smooth representations \(W'\), \(W\) and~\(W''\) of~\(G\) on
  \(R\)\nb-modules and \(G\)\nb-equivariant \(R\)\nb-module maps
  \(i\colon W'\to W\) and \(p\colon W\to W''\) such that \(i\) is
  injective, \(p\) is surjective, and \(i(W')\) is the kernel
  of~\(p\).  This extension is \emph{c-split} if there is a
  \(U\)\nb-equivariant \(R\)\nb-module map \(s\colon W''\to W\) with
  \(p\circ s=\Id_{W''}\) for each compact subgroup \(U\subseteq G\),
  and \emph{cmc-split} if the same happens for each subgroup
  \(U\subseteq G\) for which \(U/(Z(G)\cap U)\) is compact.
\end{definition}

\begin{definition}
  \label{def:rel_projective}
  An object~\(V\) of \(\Rep_R(G)\) is \emph{projective},
  \emph{c-projective} or \emph{cmc-projective} if the functor
  \(\Hom_{R,G}(V,\blank)\) is exact on all extensions, all c-split
  extensions or all cmc-split extensions in \(\Rep_R(G)\),
  respectively.  It is \emph{injective}, \emph{c-injective} or
  \emph{cmc-injective} if \(\Hom_{R,G}(\blank,V)\) is exact on the
  appropriate extensions.
\end{definition}

\begin{theorem}
  \label{the:cmc-projective}
  Cuspidal representations in \(\Rep_R(G)\) are cmc-projective.
\end{theorem}

\begin{theorem}
  \label{the:enough_c-projective}
  For any cuspidal representation~\(W''\) in~\(\Rep_R(G)\) there is a
  c-split extension \(W'\into W\prto W''\) with cuspidal and
  c-projective~\(W\).
\end{theorem}

We are going to prove Theorems \ref{the:cmc-projective}
and~\ref{the:enough_c-projective} in Sections \ref{sec:cmc-projective}
and~\ref{sec:enough_c-projective}, respectively.  In
Section~\ref{sec:good_characteristic_projective}, we specialise to the
case where~\(R\) is a good field and deduce an orthogonality result
for cuspidal and parabolically induced representations.

\subsection{Cuspidal representations are cmc-projective}
\label{sec:cmc-projective}

We are going to prove Theorem~\ref{the:cmc-projective} about
cuspidal representations being cmc-projective.  Let~\(V\) be a
cuspidal representation, so it verifies the condition in
Proposition~\ref{pro:cuspidal_criterion}.  Fix \(n\in\N\).  Let
\(e^n_\sigma V\subseteq V\) for \(\sigma\in\Buil\) be the image
of~\(e^n_\sigma\) on~\(V\).  Let
\[
\Pco^n \defeq \bigoplus_{\sigma\in\Buil} e^n_\sigma V.
\]
Elements of~\(\Pco^n\) are functions \(\psi\colon \Buil\to V\) with
finite support and \(e^n_\sigma \psi(\sigma)= \psi(\sigma)\) for all
\(\sigma\in\Buil\).  We let~\(G\) act on~\(\Pco^n\) by
\((g\psi)(\sigma) \defeq g\cdot (\psi(g^{-1}\sigma))\) for \(g\in G\),
\(\sigma\in\Buil\), \(\psi\in \Pco^n\).  This belongs
to~\(\Pco^n\) again because \(e^n_{g\sigma} = g e^n_\sigma
g^{-1}\), and it defines a smooth representation, so
\(\Pco^n\in\Rep_R(G)\).

\begin{lemma}
  \label{lem:rel_projective}
  The smooth representation~\(\Pco^n\) is cmc-projective.
\end{lemma}

\begin{proof}
  We may decompose~\(\Pco^n\) as a sum over the finite set of
  \(G\)\nb-orbits \(G\backslash\Buil\), where each summand is the
  subspace~\(\Pco_\sigma^n\) of~\(\Pco^n\) of functions supported
  on \(G\sigma\cong G/\StabS_\sigma\).  When we apply compact
  induction from~\(\StabS_\sigma\) to~\(G\) to the representation
  of~\(\StabS_\sigma\) on the \(\StabS_\sigma\)\nb-invariant subspace
  \(e^n_\sigma V\subseteq V\), we get the
  representation~\(\Pco_\sigma^n\).  Since compact induction for the
  open subgroup \(\StabS_\sigma\subseteq G\) is left adjoint to
  restriction, we get
  \[
  \Hom_G(\Pco_\sigma, \blank) \cong
  \Hom_{\StabS_\sigma}(e^n_\sigma V,\blank).
  \]
  Since the groups \(\StabS_\sigma/Z(G)\) are compact, this functor is
  exact on cmc-split extensions.
\end{proof}

The inclusion maps \(e^n_\sigma V\hookrightarrow V\) define a
\(G\)\nb-equivariant \(R\)\nb-module homomorphism \(\pi^n\colon
\Pco^n \to V\).  Since~\(V\) is cuspidal, we also get a map in the
opposite direction:

\begin{lemma}
  \label{lem:alpha}
  Let~\(V\) be a cuspidal representation.  There is a
  \(G\)\nb-equivariant \(R\)\nb-module homomorphism \(\alpha^n\colon
  V\to \Pco^n\) defined by
  \[
  \alpha^n(v)(\sigma) \defeq (-1)^{\dim(\sigma)} e^n_\sigma(v)
  \qquad\text{for }v\in V,\ \sigma\in\Buil.
  \]
  The map \(\pi^n\circ \alpha^n\colon V\to V\) is idempotent with
  image \(\sum_{x\in\Buil^0} e^n_x V\).
\end{lemma}

\begin{proof}
  Let \(v\in V\) and let \(\sigma\in\Buil\).
  Proposition~\ref{pro:cuspidal_criterion} shows that the set of
  \(g\in G\) with \(e^n_{g\sigma} v = g\av{U_\sigma^n} g^{-1} v
  \neq0\) is compact modulo \(Z(G)\).  Hence \(\alpha^n(v)(\tau)\)
  vanishes on all but finitely many \(\tau\in G\sigma\).  Since
  \(G\backslash \Buil\) is finite, \(\alpha^n(v)\) has finite
  support.  We have \(\alpha^n(v)(\sigma)\in e^n_\sigma V\) by
  construction, so \(\alpha^n(v)\in \Pco^n\).  The map
  \(\alpha^n\colon V\to \Pco^n\) is \(R\)\nb-linear and
  \(G\)\nb-equivariant because \(e^n_{g\sigma} = g e^n_\sigma
  g^{-1}\).  Theorem~\ref{the:support_projections} implies that
  \(\pi^n\circ \alpha^n\colon V\to V\) is idempotent with the
  asserted image: write~\(\Buil\) as a union of finite convex
  subcomplexes and use that all sums that occur are finite.
\end{proof}

If~\(V\) is finitely generated, then \(V=\sum_{x\in\Buil^0} e^n_x
V\) for some \(n\ge0\), so that \(\pi^n\circ\alpha^n=\Id_V\)
and~\(V\) is a direct summand of~\(\Pco^n\).  For general~\(V\), we
let \(\Pco\defeq \bigoplus_{n=0}^\infty \Pco^n\).  This is
cmc-projective by Lemma~\ref{lem:rel_projective}.  Let \(\pi\colon
\Pco\to V\) be induced by the maps \(\pi^n\colon \Pco^n\to V\).
Define \(\alpha\colon V\to \Pco\) by \(\alpha(v) =
(\alpha(v)_n)_{n\in\N}\) with
\[
\alpha(v)_n \defeq \alpha^n\circ (\Id_V-\pi^{n-1}\circ\alpha^{n-1})(v)
\]
for \(n\in\N\), with the convention \(\pi^{-1}\circ\alpha^{-1}=0\) for
\(n=0\).  For any \(v\in V\) and \(\sigma\in\Buil^0\), there is
\(n_0\ge0\) with \(e^n_\sigma v=v\) for all \(n\ge n_0\).  This
implies \((\Id_V-\pi^n\circ\alpha^n)(v)=0\) and hence
\(\alpha(v)_n=0\) for \(n\ge n_0\).  Thus \(\alpha(v)\in \Pco\).
Since the maps \(\alpha^n\) and~\(\pi^n\) are \(G\)\nb-equivariant
\(R\)\nb-module homomorphisms, so is~\(\alpha\).  The sequence of
idempotent maps \(\pi^n\circ\alpha^n\) on~\(V\) is increasing,
that is, \(\pi^n\circ\alpha^n\circ\pi^{n-1}\circ\alpha^{n-1} =
\pi^{n-1}\circ\alpha^{n-1}\).  Hence
\[
\pi\circ\alpha(v)
= \sum_{n=0}^\infty
(\pi^n\circ\alpha^n-\pi^{n-1}\circ\alpha^{n-1})(v)
= v.
\]
This finishes the proof that cuspidal representations are
cmc-projective.

\subsection{Enough c-projective cuspidal representations}
\label{sec:enough_c-projective}

We are going to prove Theorem~\ref{the:enough_c-projective} about
cuspidal representations being c-split quotients of c-projective
cuspidal representations.  An important tool that we will also use
for other purposes is the subgroup \(\Gz\subseteq G\) that is the
intersection of the kernels of all unramified characters.  The
following is proved in \cite{Renard:Representations}*{Section
  V.2.3--6}:

\begin{proposition}
  \label{pro:Gz}
  The subgroup~\(\Gz\) is open and normal in~\(G\) and contains any
  compact subgroup of~\(G\).  The quotient group~\(G/\Gz\) is free
  Abelian of finite rank.  The image of~\(Z(G)\) in~\(G/\Gz\) has
  finite index, and \(\Gz[Z] \defeq Z(G)\cap \Gz\) is compact.\qed
\end{proposition}

The subgroup~\(\Gz\) is equal to the subgroup generated by the
compact elements of~\(G\) used in~\cite{Vigneras:l-modulaires}.  We
modify the criterion for cuspidal representations in
Proposition~\ref{pro:cuspidal_criterion}:

\begin{proposition}
  \label{pro:cuspidal_criterion_circ}
  A smooth representation \(V\in\Rep_R(G)\) is cuspidal if and only
  if for each \(v\in V\) and \(x\in\Hecke\), the set of \(g\in \Gz\)
  with \(x g v\neq0\) is compact.
\end{proposition}

\begin{proof}
  It makes no difference whether we use all \(x\in\Hecke\) or only
  \(\av{U}\) for compact open pro-\(p\) subgroups \(U\subseteq G\)
  because for any \(x\in\Hecke\) there is~\(U\) with \(x=x\av{U}\).
  Assume first that for each \(v\in V\), the set of \(g\in \Gz\) with
  \(x g v\neq0\) is compact.  Let \(h_1,\dotsc,h_n\in G\) be
  representatives for the finite quotient group \(G/\Gz Z(G)\).  So
  every element of~\(G\) is of the form \(g z h_i\) with \(g\in\Gz\),
  \(z\in Z(G)\), \(i\in\{1,\dotsc,n\}\).  For each~\(i\), the
  set~\(K_i\) of \(g\in\Gz\) with \(x g h_i v\neq0\) is compact.  The
  set of \(g\in G\) with \(x g v\neq0\) is \(\bigsqcup_{i=1}^n K_i h_i
  Z(G)\), which is compact modulo the centre.  Hence~\(V\) verifies
  the criterion for being cuspidal in
  Proposition~\ref{pro:cuspidal_criterion}.  Conversely, if~\(V\)
  verifies that criterion, then the set of \(g\in\Gz\) with \(x g
  v\neq0\) has to be compact because~\(\Gz[Z]\) is compact.
\end{proof}

We use the criterion in Proposition~\ref{pro:cuspidal_criterion_circ}
to define which smooth representations of~\(\Gz\) are cuspidal.  Then
Proposition~\ref{pro:cuspidal_criterion_circ} says that a
representation of~\(G\) is cuspidal if and only if its restriction
to~\(\Gz\) is cuspidal.

\begin{theorem}
  \label{the:c-projective_Gz}
  Cuspidal representations of~\(\Gz\) are c-projective
  in~\(\Rep_R(\Gz)\).
\end{theorem}

\begin{proof}
  Since \(\Gz[Z]\) is compact, cmc- and c-projectives
  in~\(\Rep_R(\Gz)\) are the same.  Hence the assertion follows as in
  the proof of Theorem~\ref{the:cmc-projective}.
\end{proof}

The functor that restricts a smooth \(G\)\nb-representation
to~\(\Gz\) and then applies compact induction to~\(G\) maps \(W''\in
\Rep_R(G)\) to \(W = \bigoplus_{[g]\in G/\Gz} W''\), the space of
finitely supported functions \(G/\Gz\to W''\), with the (smooth)
representation of~\(G\) by \((g\psi)[h] = g \cdot (\psi[g^{-1} h])\)
for \(g,h\in G\) and \(\psi\colon G/\Gz \to W''\); here we use
that~\(G\) already acts on~\(W''\) to simplify the result.  Since
this action of~\(G\) on~\(W\) restricts to the given representation
of~\(\Gz\) on~\(W''\) on each summand,
Proposition~\ref{pro:cuspidal_criterion_circ} shows that~\(W\) is
cuspidal if and only if~\(W''\) is.

Define \(\pi\colon W\to W''\), \(\pi(\psi) = \sum_{[g]\in G/\Gz}
\psi[g]\).  This map is \(G\)\nb-equivariant and \(R\)\nb-linear.  It
splits by the \(\Gz\)\nb-equivariant \(R\)\nb-linear map \(s\colon
W''\to W\) defined by \(s(v)[1] \defeq v\) and \(s(v)[g] \defeq 0\)
for \([g]\neq[1]\) in~\(G/\Gz\).  Since all compact subgroups are
contained in~\(\Gz\), \(\pi\) is a c-split surjection.  The
adjointness between compact induction and restriction for the open
subgroup \(\Gz\subseteq G\) gives
\[
\Hom_G(W,\blank) \cong \Hom_{\Gz}(W'',\blank).
\]
This functor is exact on c-split extensions by
Theorem~\ref{the:c-projective_Gz}
because~\(W''\) is cuspidal.  Thus~\(W''\) is c-projective and
cuspidal.  Take \(W'\defeq \ker(\pi)\) to finish the proof of
Theorem~\ref{the:enough_c-projective}.

\subsection{Representations over fields of good characteristic}
\label{sec:good_characteristic_projective}

Now let~\(R\) be a \emph{good field}, that is, its characteristic does
not divide the pro-order of~\(G\).

\begin{corollary}
  \label{cor:projective_field}
  Let~\(R\) be a good field.  Cuspidal representations of~\(G\) are
  projective as representations of~\(\Gz\), and quotients of
  projective, cuspidal representations of~\(G\).
\end{corollary}

\begin{proof}
  Any extension of vector spaces over the field~\(R\) splits, and we
  can make the section \(U\)\nb-invariant for a given compact subgroup
  \(U\subseteq G\) by averaging over the normalised Haar measure
  of~\(U\); this measure has values in~\(R\) because \([U:U']\) is
  invertible in~\(R\) for any open subgroup \(U'\subseteq U\) by our
  assumption on the characteristic of~\(R\).  Thus all extensions are
  c-split exact, and c-projective objects are projective.  Now
  everything follows from Theorems \ref{the:c-projective_Gz}
  and~\ref{the:enough_c-projective}.
\end{proof}

\begin{definition}
  \label{def:induced}
  A representation is \emph{subinduced} if it is contained in a direct
  sum of parabolically induced representations from proper parabolic
  subgroups.
\end{definition}

\begin{theorem}
  \label{the:cuspidal_induced_orthogonal}
  Let~\(R\) be a good field.  Let \(V\in\Rep_R(G)\) be cuspidal and
  let \(W\in\Rep_R(G)\) be subinduced.  Any map \(W\to V\) or \(V\to
  W\) vanishes.  Subinduced representations have no non-zero cuspidal
  subquotients.
\end{theorem}

\begin{proof}
  The vanishing of maps \(V\to W\) is well-known and follows from the
  (easy) left adjointness of parabolic restriction to parabolic
  induction (see \cite{Vigneras:l-modulaires}*{II.2.3} or
  \cite{Renard:Representations}*{Section VI.7.2}).  Assume there were
  a non-zero map \(f\colon W\to V\).  Since subrepresentations of
  cuspidal representations remain cuspidal, \(\im(f)\) is cuspidal.
  By Corollary~\ref{cor:projective_field}, \(\im(f)\) is a quotient of
  a projective, cuspidal representation, \(p\colon V'\prto \im(f)\).
  Since~\(V'\) is projective, there is a map \(h\colon V'\to W\) with
  \(f\circ h=p\).  Since \(\im(f)\neq0\), we have \(p\neq0\) and thus
  \(h\neq0\).  But since~\(V'\) is cuspidal and~\(W\) is subinduced,
  there is no non-zero map \(V'\to W\).  Thus there cannot be a
  non-zero map \(W\to V\).

  If subinduced representations may have non-zero cuspidal
  subquotients,
  they also may have non-zero cuspidal quotients.  This would give a
  non-zero map from a subinduced representation to a cuspidal
  representation, which is impossible.
\end{proof}

An example mentioned in \cite{Vigneras:l-modulaires}*{II.2.5}
shows that parabolically induced representations of~\(\Gl(2,\Q_5)\)
over the field with \(3\)~elements may have cuspidal subquotients.
Hence the assertions of Theorem~\ref{the:cuspidal_induced_orthogonal}
and Corollary~\ref{cor:projective_field} become false in this case.
Since~\(3\) divides the pro-order of~\(\Gl(2,\Q_5)\), this is no
contradiction.

\begin{remark}
  Vign\'eras~\cite{Vigneras:Cohomology} proves that
  \emph{irreducible} cuspidal representations are projective in the
  category of representations with a fixed central character; this
  is essentially the same as proving that they are projective as
  \(\Gz\)\nb-representations.  The proof depends on Schur's Lemma,
  which implies that any non-zero map on an irreducible
  representation is invertible.  There are only finitely many
  isomorphism classes of cuspidal representations of~\(\Gz\) that
  contain a \(K\)\nb-fixed vector for any fixed compact subgroup
  \(K\subseteq \Gz\).  This deep fact and the projectivity of
  irreducible cuspidal representations imply a product decomposition
  of the category of representations of~\(\Gz\) into one factor for
  each irreducible cuspidal representation and one factor for those
  representations that have no cuspidal subquotients (compare
  \cite{Renard:Representations}*{Section VI.3.4}).  This product
  decomposition shows that \emph{all} cuspidal representations are
  projective as representations of~\(\Gz\) because this holds for
  irreducible cuspidal representations.  In contrast, we directly
  prove that arbitrary cuspidal representations are projective and
  also injective as representations of~\(\Gz\).  Much of the direct
  product decomposition in \cite{Renard:Representations}*{Section
    VI.3.4} follows from such projectivity and injectivity results,
  compare Theorem~\ref{the:split_rep} below.
\end{remark}

\section{Relative injectivity of cuspidal representations}
\label{sec:rel_injective}

\begin{definition}
  \label{def:uniformly_cuspidal}
  A representation~\(V\) is \emph{uniformly cuspidal} if for each
  \(x,y\in\Hecke\) the set of \(g\in G\) for which \(x g y\in\Hecke\)
  acts non-trivially on~\(V\) is compact modulo~\(Z(G)\).
\end{definition}

It is shown in~\cite{Vigneras:l-modulaires}*{II.2.16} that cuspidal
representations are uniformly cuspidal.  This allows us to prove
relative injectivity results for cuspidal representations.

\subsection{Cmc-injectivity}
\label{sec:cmc-injective}

\begin{theorem}
  \label{the:cmc-injective}
  Cuspidal representations in \(\Rep_R(G)\) are cmc-injective.
\end{theorem}

For \(n\ge0\), let
\[
\tilde{I}^n \defeq \prod_{\sigma\in\Buil} e^n_\sigma V.
\]
Elements of~\(\tilde{I}^n\) are functions \(\psi\colon \Buil\to V\)
with \(e^n_\sigma \psi(\sigma) = \psi(\sigma)\) for all
\(\sigma\in\Buil\), and~\(G\) acts on such functions as above, by
\((g\psi)(\sigma) \defeq g\cdot (\psi(g^{-1}\sigma))\) for \(g\in G\),
\(\sigma\in\Buil\), \(\psi\in \tilde{I}^n\).  This representation is
not smooth.  Let \(I^n\subseteq \tilde{I}^n\) be the subspace of
smooth elements, that is, elements fixed by some compact open subgroup
of~\(G\).

\begin{lemma}
  \label{lem:rel_injective_I}
  The smooth representation~\(I^n\) is cmc-injective.
\end{lemma}

\begin{proof}
  This is proved in the same way as Lemma~\ref{lem:rel_projective}.
  Let~\(I^n_\sigma\) be the restriction of~\(I^n\) to the orbit
  of~\(\sigma\).  Then \(I^n \cong \prod_{\sigma\in
    G\backslash\Buil} I^n_\sigma\) because \(G\backslash\Buil\) is
  finite.  Thus~\(I^n\) is cmc-injective if and only
  if~\(I^n_\sigma\) is cmc-injective for all \(\sigma\in
  G\backslash\Buil\).  The induction functor from~\(G_\sigma\)
  to~\(G\) maps~\(V\) to~\(I^n_\sigma\) and is right adjoint to the
  restriction functor.  Thus the same argument as in the proof of
  Lemma~\ref{lem:rel_projective} shows that~\(I^n_\sigma\) is
  cmc-injective.
\end{proof}

Since \(\Pco^n\subseteq \tilde{I}^n\) is a smooth representation,
we have \(\Pco^n\subseteq I^n\).  Thus the map~\(\alpha^n\)
constructed in Section~\ref{sec:cmc-projective} is also a
\(G\)\nb-equivariant \(R\)\nb-module map \(\alpha^n\colon V\to
I^n\).

\begin{lemma}
  \label{lem:extend_pi}
  If~\(V\) is cuspidal, then the map \(\pi^n\colon \Pco^n\to V\)
  extends to a \(G\)\nb-equivariant \(R\)\nb-module map
  \(\bar\pi^n\colon I^n\to V\).
\end{lemma}

\begin{proof}
  Let \(\psi\in I^n\subseteq\tilde{I}^n\).  Since~\(\psi\) is a
  smooth element, it is fixed by some compact open pro-\(p\)
  subgroup~\(U\) of~\(G\).  Let \(\sigma\in\Buil\).  The set of \(g\in
  G\) with \(\av{U} g e^n_\sigma\neq0\) is compact modulo~\(Z(G)\)
  because~\(V\) is uniformly cuspidal (see
  \cite{Vigneras:l-modulaires}*{II.2.16}).  Since
  \(\psi(g\sigma)\in e^n_{g\sigma} V = g e^n_{\sigma} g^{-1} V = g
  e^n_{\sigma} V\), the set of \(g\in G\) with
  \(\av{U}\psi(g\sigma)\neq0\) is compact modulo~\(Z(G)\).  Thus
  \(\av{U}\psi(g\sigma)=0\) for all but finitely many elements in the
  orbit~\(G\sigma\) because the stabiliser~\(\StabS_\sigma\)
  of~\(\sigma\) contains~\(Z(G)\).  Since \(G\backslash\Buil\) is
  finite, the sum
  \[
  \bar\pi^n_U(\psi) \defeq \sum_{\sigma\in\Buil} \av{U}\psi(\sigma)
  \]
  is finite.  If~\(\psi\) has finite support, then we may
  pull~\(\av{U}\) out of the sum and get \(\bar\pi^n_U(\psi) =
  \av{U}\pi^n(\psi) = \pi^n(\av{U}\psi) = \pi^n(\psi)\)
  because~\(\pi^n\) is \(G\)\nb-equivariant.

  We claim that~\(\bar\pi^n_U\) does not depend on~\(U\), so we get
  a well-defined map \(\bar\pi^n\colon I^n\to V\).  Taking this
  for granted, it is routine to check that~\(\bar\pi^n\) is a
  \(G\)\nb-equivariant \(R\)\nb-module homomorphism.  We checked
  already that it extends~\(\pi^n\), so this finishes the proof.

  It remains to prove that~\(\bar\pi^n\) does not depend on~\(U\).
  If \(U_1,U_2\) are compact open pro-\(p\) subgroups fixing~\(\psi\),
  then there is an open subgroup \(U_3\subseteq U_1\cap U_2\) which is
  normal in~\(U_1\).  And~\(U_3\) contains an open subgroup~\(U_4\)
  that is normal in~\(U_2\) and hence in~\(U_3\).  Thus we are done if
  we show that \(\bar\pi^n_U = \bar\pi^n_{U'}\) if~\(U\) is a
  normal open subgroup of~\(U'\).  Then \(\av{U'} =
  [U':U]^{-1}\sum_{g\in U'/U} \av{U}g\).  We compute
  \begin{align*}
    \bar\pi^n_{U'}(\psi)
    &= \frac{1}{[U':U]}\sum_{\twolinesubscript{g\in U'/U}{\sigma\in\Buil}}
    \av{U}g\cdot (\psi(\sigma))
    = \frac{1}{[U':U]}\sum_{\twolinesubscript{g\in U'/U}{\sigma\in\Buil}}
    \av{U}g\cdot (\psi(g^{-1}\sigma))
    \\&= \frac{1}{[U':U]}\sum_{\twolinesubscript{g\in U'/U}{\sigma\in\Buil}}
    \av{U} (g\psi)(\sigma)
    = \bar\pi^n_U(\psi);
  \end{align*}
  the second step reindexes the sum over~\(\sigma\); the third step is
  the definition of the \(G\)\nb-action on~\(I^n\), and the last
  step uses \(g\psi=\psi\) for \(g\in U'\).
\end{proof}

Lemma~\ref{lem:alpha} shows that \(\bar\pi^n\circ \alpha^n =
\pi^n\circ \alpha^n\) is idempotent with image
\(\sum_{x\in\Buil^0} e^n_x V\).  If \(V=\sum_{x\in\Buil^0} e^n_x
V\), then Lemmas \ref{lem:alpha} and~\ref{lem:extend_pi} show
that~\(V\) is cmc-injective because it is a direct summand in the
cmc-injective smooth representation~\(I^n\).

In general, let \(V^n\defeq
(\pi^n\alpha^n-\pi^{n-1}\alpha^{n-1}) V\) for \(n\in\N\).  These
are direct summands in~\(V\) with \(V=\bigoplus_{n\in\N} V^n\).  We
claim that any element of \(\prod_{n\in\N} V^n\) that is smooth for
the \(G\)\nb-action already belongs to the direct sum
\(\bigoplus_{n\in\N} V^n\).  Indeed, a smooth vector is in the image
of~\(e^{m}_x\) for some \(m\in\N\) because these form an approximate
unit, and then it is killed by
\(\pi^n\alpha^n-\pi^{n-1}\alpha^{n-1}\) for \(n>m\).  Hence~\(V\)
is also the \emph{product} of the~\(V^n\) in the category of smooth
representations.  The proof above shows that the factors~\(V^n\) are
cmc-injective.  So is~\(V\) because products of cmc-injective
representations remain cmc-injective.  This finishes the proof of
Theorem~\ref{the:cmc-injective}.

\subsection{Enough c-injective cuspidal representations}
\label{sec:enough_c-injective}

\begin{theorem}
  \label{the:enough_c-injective}
  For any cuspidal representation~\(W'\) in~\(\Rep_R(G)\) there is a
  c-split extension \(W'\into W\prto W''\) with cuspidal and
  c-injective~\(W\).
  Cuspidal representations of~\(\Gz\) are c-injective
  in~\(\Rep_R(\Gz)\).
\end{theorem}

\begin{proof}
  This is analogous to the proofs of Theorem
  \ref{the:enough_c-projective} and~\ref{the:c-projective_Gz}.
  For the first statement, we restrict the
  representation on~\(W'\) to~\(\Gz\) and now apply ordinary
  induction to define~\(W\); this produces a c-injective
  representation of~\(G\) because induction is right adjoint to
  restriction.  There is a canonical \(G\)\nb-equivariant map \(W'
  \to W\), which splits \(\Gz\)\nb-equivariantly by restriction to
  \(\Gz\subseteq G\).  Letting \(W''\defeq W/W'\), we get the
  desired c-split extension.
\end{proof}

\begin{corollary}
  \label{cor:banal_field_injective}
  Let~\(R\) be a good field.  Any cuspidal representation
  in~\(\Rep_R(G)\) is contained in a cuspidal, injective
  representation.
\end{corollary}

\begin{proof}
  See the proof of Corollary~\ref{cor:projective_field}.
\end{proof}

\begin{theorem}
  \label{the:split_rep}
  Let~\(R\) be a good field.  The category~\(\Rep_R(G)\) is the
  product of the two subcategories of cuspidal and subinduced
  representations.
\end{theorem}

\begin{proof}
  Theorem~\ref{the:cuspidal_induced_orthogonal} shows that there are
  no arrows in either direction between the subcategories of cuspidal
  and subinduced representations.  It remains to show that every
  representation is a product of a cuspidal and a subinduced
  representation.

  Let~\(S\) be the set of proper standard parabolic subgroups
  of~\(G\).  Let \(i_P\) and~\(r_P\) for \(P\in S\) be the parabolic
  induction and restriction functors.  The right adjointness
  of~\(i_P\) to~\(r_P\) gives natural maps \(\beta_P\colon V\to i_P
  r_P(V)\) for \(V\in\Rep_R(G)\), which we put together into a natural
  map \(\beta\colon V\to \bigoplus_{P\in S} i_P r_P(V)\).  Let
  \(V'\defeq \ker(\beta)\), \(V''\defeq \im(\beta)\).  These form an
  extension \(V'\into V\prto V''\) in~\(\Rep_R(G)\) with
  subinduced~\(V''\).  Inspection shows that \(\beta(v)=0\) for \(v\in
  V\) if and only if for each unipotent subgroup \(N\subseteq G\)
  there is a compact open subgroup~\(U_N\) with \(\av{U_N} v=0\);
  thus~\(V'\) is cuspidal.

  We may embed~\(V'\) in an injective, cuspidal representation~\(W\)
  by Corollary~\ref{cor:banal_field_injective}.  The embedding
  \(i\colon V' \to W\) extends to a map \(\varphi\colon V\to W\)
  because~\(W\) is injective.  Since \(\varphi(V') = i(V')\),
  \(\varphi\) induces a map \(\dot\varphi\colon V''\to W/V'\).  The
  quotient~\(W/V'\) is again cuspidal, so
  Theorem~\ref{the:cuspidal_induced_orthogonal} gives
  \(\dot\varphi=0\).  Thus \(\varphi(V)\subseteq i(V')\),
  so~\(\varphi\colon V\to V'\) is a section for our extension.  Hence
  \(V\cong V'\oplus V''\).
\end{proof}

\end{document}